\documentclass[12pt]{amsart}
\usepackage{amsmath}
\allowdisplaybreaks[1] 
\usepackage{amssymb}

\newtheorem{thm}{Theorem}[section]

\newtheorem{lemma}[thm]{Lemma}
\newtheorem{prop}[thm]{Proposition}
\theoremstyle{definition}
\newtheorem{defn}[thm]{Definition}
\newtheorem{rem}[thm]{Remark}

\begin{document}
\title{Blocking: new examples and properties of products }
\author{Pilar Herreros}
\address{Department of Mathematics, University of Pennsylvania, Philadelphia, PA 19104, USA}
\email{pherrero@math.upenn.edu}
\begin{abstract}
We say that a pair of points $x$ and $y$ is secure if there exist a finite set of blocking points such
that any geodesic between $x$ and $y$ passes through one of the blocking points. The main point of this
paper is to exhibit new examples of blocking phenomena both in the manifold and the billiard table
setting. As an approach to this, we study if the product of secure configurations (or manifolds) is
also secure.

We introduce the concept of \emph{midpoint security} that imposes that the geodesic reaches a blocking
point exactly at its midpoint. We prove that products of midpoint secure configurations are midpoint
secure. On the other hand, we give an example of a compact $C^1$ surface that contains secure
configurations that are not midpoint secure. This surface provides the first example of an insecure
product of secure configurations, as well as billiard table examples.
\end{abstract}
\maketitle

\section{Introduction}

We say that a pair of points $x$ and $y$ is secure if there exist a finite set of blocking points such
that any geodesic between $x$ and $y$ passes through one of the blocking points. The study of this
property originated in the context of polygonal billiards and translation surfaces (see e.g.
\cite{Gu04}, \cite{Gu06}, \cite{GJ}, \cite{Mo05}). It has an interpretation in geometric optics, where
being secure means that one point is shaded from the light emanated by the other point by finitely many
point blocks. Or that they can't "see" each other. Another interpretation relates to the name security,
as we can think of two points being secure if any path between them passes through one of finitely many
check points.

Recently, security has been studied on Riemannian manifolds. In particular for compact, locally
symmetric spaces (see \cite{GS}), and its relation with entropy (see \cite{BG}, \cite{LS}) among
others. Security is a rare phenomenon. The emphasis in the area is to show that small blocking sets
imply very strong conditions on the geometry of the manifold. For instance, on all known examples of
security the blocking occurs at the midpoints of geodesics. We mention some other such results and
conjectures below. The main point of this paper is to exhibit new examples of blocking phenomena both
in the manifold and the billiard table setting. Here the manifolds will be closed $C^1$ smooth
Riemannian manifolds, while the billiard tables we consider are compact $C^1$ smooth Riemannian
manifolds with boundary (with and without corners).

The search for examples and for an understanding of its relationship with other aspects of Riemannian
geometry raises the question of how this property behaves under products of manifolds. This question
was brought to my attention by K. Burns and E. Gutkin, who partially solved it in \cite{BG} where they
prove that if a (configuration on a) product manifold is secure, then so is (its projection to) each
factor. In this paper we study the converse of this result, more precisely if the product of secure
configurations (or manifolds) is also secure.

In section \ref{secMS} we introduce the concept of \emph{midpoint security} that imposes that the
geodesic reaches a blocking point exactly at its midpoint. We discuss its relation with security and
prove that products of midpoint secure configurations are midpoint secure. On the other hand, we give
some cases when this condition is actually necessary to achieve security. In section \ref{secEx} we
build an example of a compact $C^1$ surface that contains secure configurations that are not midpoint
secure. This new example of blocking phenomena also supplies the first example of an insecure product
of secure configurations. We finish the paper using this surface to give examples of non-planar
billiard tables with secure configurations that are not midpoint secure.\\

I would like to give my special thanks to K. Burns and E. Gutkin for their interest in this work.

\section{Security and Midpoint Security}\label{secMS}

Throughout this paper geodesics are assumed to have positive finite length and to be parametrized by
$[0,1]$ proportional to arclength. The length $0$ case, when the geodesic is a point, will not be
considered a geodesic and when needed will be refered to as a \emph{constant path}.

If $\gamma :[0,1]\to M$ is a geodesic, the \emph{endpoints} of $\gamma$ are its \emph{initial point}
$\gamma(0)$ and \emph{final point} $\gamma(1)$. The points $\gamma(t)$ with $t\in (0,1)$ are
\emph{interior points} of $\gamma$. We will say that $\gamma$ \emph{passes through} a point $x\in M$ if
$x$ is an interior point of $\gamma$.\\

A \emph{configuration} in $M$ is an ordered pair of points in $M$, these points may coincide. For a
configuration $(x,y)$ in $M$ we say that a geodesic $\gamma$ \emph{joins} $x$ to $y$ if it has initial
point $x$ and final point $y$. We will denote by $G(x,y)$ the set of all geodesics joining $x$ and $y$.

A set $B$ is a \emph{blocking set} for a collection of geodesics if every geodesic in the collection
passes through a point in $B$.\\

\begin{defn}
A configuration $(x,y)$ is \emph{secure} if the collection $G(x,y)$ of geodesics joining $x$ and $y$
has a finite blocking set. Otherwise the configuration is \emph{insecure}.\\
A Riemannian manifold is \emph{secure} if all configurations in it are secure.
\end{defn}
\smallskip
Note that any geodesic in $G(x,y)$ has to contain a segment that joins $x$ to $y$ without passing
through either endpoint, and this segment is also contained in $G(x,y)$. Blocking this segment will
also block the original geodesic, so we can (and will) always choose the blocking set to be in
$M\setminus \{x,y\}$. This definition is then equivalent to the one used in \cite{BG}, where they only
consider geodesics that don't pass through the endpoints.\\

The first to consider the relation between security and product manifolds were E. Gutkin and V.
Schr\"oder, in \cite{GS} while studying security of locally symmetric spaces they proved that if a
product manifold is secure, so are its factors. In \cite{BG} K. Burns and E. Gutkin proved the
following lemma and used it to give examples of totally insecure manifolds.\\

\begin{lemma}\label{lemmaBG}
If a configuration in a product manifold is secure, then the projection to each factor is secure.

Moreover; if the configuration has blocking set $B$, then the projection of $B$ minus the endpoints is
a blocking set for the projection in each factor.
\end{lemma}

To study the converse of this lemma we introduce a related but stronger property.
We will say that a set $B$ is a \emph{midpoint blocking set} for a collection of geodesics if every
geodesic in the collection has its midpoint in $B$.\\

\begin{defn}
A configuration $(x,y)$ is \emph{midpoint secure} if the collection $G(x,y)$ of joining geodesics has a finite midpoint blocking set $B$.\\
A Riemannian manifold is \emph{midpoint secure} if all configurations in it are midpoint secure.
\end{defn}

\smallskip

Unlike the security case, we will allow the endpoints $x$ and $y$ to be in $B$. This is actually
necessary in some cases since, for instance, a simple closed geodesic travelled twice has midpoint
equal to the endpoints.\\

Note that any midpoint secure configuration is also secure, with the same blocking set minus the
endpoints if necessary. On the other hand, all previously known examples of secure configurations on
Riemannian manifolds are midpoint secure. In particular, this can be seen for all configurations in
locally symmetric spaces of euclidean type in \cite{GS}, where the blocking set given is actually
midpoint blocking. This is also true for arithmetic polygonal billiards and translation surfaces (see
\cite{Gu06},\cite{GJ}).

One of the main goals of this area is to characterize certain manifolds by their blocking properties.
The only known secure compact manifolds are flat, and it has been conjectured that these are in fact
the only ones among closed smooth Riemannian manifolds. For more details see K. Burns and E. Gutkin
\cite{BG} and J.-F. Lafont and B. Schmidt \cite{LS}, where they give the conjecture together with some
partial results in this direction.

Another case is that of compact rank one symmetric spaces (CROSS). These manifolds have the property
that any pair of distinct points that are not at distance equal to the diameter are secure, with a
blocking set consisting of only two points. It was conjectured in \cite{LS} that the CROSSes are the
only compact Riemannian manifolds with this property. Part of the evidence they provide is proving that
all Blaschke manifolds have this property. They also conjecture that the only such manifold with the
additional property that for any point $x$, $G(x,x)$ can be blocked by a single point is the round
sphere. This conjecture was later proved by B. Schmidt and J. Souto in \cite{SS}.\\

All the manifolds mentioned above satisfy the corresponding blocking properties if we replace security
by midpoint security. This can be seen in the proofs of security of each case. Moreover, if all secure
compact manifolds are flat as conjectured, then they also are all midpoint secure. Proving any of these
conjectures using midpoint security would be a significant progress. \\

This also raises the question whether these conditions are actually equivalent. This is not true, as we
will see in section \ref{secEx} where we give an example of a surface with secure configurations that
are not midpoint secure.\\ 

\begin{lemma}\label{lemma1}
A configuration in a product manifold is midpoint secure if and only if the projection to each factor
is midpoint secure.
\end{lemma}

\begin{proof} (The first half follows Burns and Gutkin's argument for Lemma \ref{lemmaBG}, we will include it here for completeness.)
Let $x=(x_1,x_2)$ and $y=(y_1,y_2)$ be two points in the product manifold $M_1\times M_2$. Any geodesic
$\gamma$ joining $x$ and $y$ projects to geodesics $\gamma_i=\pi_i(\gamma)$ in $M_i$ joining $x_i$ and
$y_i$, we will denote this by $\gamma=\gamma_1\times \gamma_2$. Note that one $\gamma_i$ might be a
constant path, but the argument works regardless.

Suppose that the configuration $(x,y)$ is midpoint secure, and let $B$ be its midpoint blocking set.
Let $B_1=\pi_1(B)$ be the projection of this set to $M_1$, we claim that this is a blocking set for
$(x_1, y_1)$. To see this take any geodesic $\sigma$ in $M_2$ from $x_2$ to $y_2$. For any geodesic
$\gamma_1 \in G(x_1,y_1)$ joining $x_1$ and $y_1$, the geodesic $\gamma_1 \times \sigma$ joins $x$ and
$y$ in $M_1\times M_2$ and therefore is midpoint blocked by $B$. This means that $\gamma_1(1/2) \times
\sigma(1/2)\in B$, so $\gamma_1(1/2)\in B_1$, and since $\gamma_1$ is arbitrary this proves that the
configuration $(x_1,y_1)$ is midpoint secure in $M_1$. Reversing the roles of $M_1$ and $M_2$ we see
that the configuration $(x_2,y_2)$ is midpoint secure in $M_2$.  \\

To prove the converse lets assume that $(x_1,y_1)$ and $(x_2, y_2)$ are midpoint secure
in $M_1$ and $M_2$ respectively, and let $B_1$ and $B_2$ be the corresponding blocking sets. If
$x_i=y_i$ add this point to $B_i$, so that the constant path joining them is also midpoint blocked by
$B_i$. Let $x=(x_1,x_2)$ and $y=(y_1,y_2)$ in the product manifold $M_1\times M_2$, we will see that
the set $B=B_1\times B_2$ is a midpoint blocking set for $G(x,y)$.  For this take any $\gamma \in
G(x,y)$ and write it in the form $\gamma=\gamma_1\times \gamma_2$ as above. Since $\gamma_i$ joins
$x_i$ and $y_i$, it is midpoint blocked by $B_i$, that is, $\gamma_i(1/2) \in B_i$ and therefore
$\gamma_1(1/2) \times \gamma_2(1/2)\in B_1\times B_2 =B$, completing the proof.
\end{proof}

From this Lemma immediately follows that:\\

\begin{prop}
A product manifold is midpoint secure if and only if each factor is midpoint secure. \\
\end{prop}

Regarding the security on product manifolds this result says that the product of midpoint secure
configurations (or manifolds) is secure. But what if one, or both, is not midpoint secure? We will
analyze the case when one of the factors is a round $S^2$ to show that in certain cases the condition
on the midpoints is needed.\\

\begin{prop}\label{prop}
Given a Riemannian manifold $M$ the following statements are equivalent:
\begin{enumerate}
\item A configuration in the product manifold $M\times S^2$ is secure if and only if the projection to
each factor is secure.
\item All secure configurations in M are midpoint secure.
\end{enumerate}
\end{prop}

\begin{proof}Assume $1)$ holds and let $(x,y)$ be a secure configuration on M. For any $p\in S^2$ in the
2-sphere let $q$ be its antipodal point. Then $(p,p)$ is (midpoint) secure with blocking set $\{p,q\}$ and
by $1)$ $((x,p),(y,p))$ is secure in $M\times S^2$. Let $B$ be the blocking set for $((x,p),(y,p))$,
and let $B_M=\pi_1(B)$ and $B_S=\pi_2(B)$ be its projection to each factor. By Lemma \ref{lemmaBG}
$B_M$ and $B_S$ are blocking sets for $(x,y)$ and $(p,p)$ respectively. Note that there are infinitely
many simple geodesics in $G(p,p)$ and any pair of them only intersect in $q$, so any finite blocking
set for $(p,p)$ has to contain $q$ and a geodesic in $G(p,p)$ that is only blocked by $q$. Let $\sigma
\in G(p,p)$ be a simple great circle that is only blocked by $q$.

For any $\gamma \in G(x,y)$ the geodesic $\gamma\times\sigma$ is in $G((x,p),(y,p))$ so it has to
pass through a blocking point $b\in B$ at some time $t_0\in (0,1)$. By the definition of $B_S$ we have
$\sigma(t_0)=\pi_2(b)\in B_S$, but $\sigma$ is only blocked by $q$ at time $1/2$. Therefore
$\pi_2(b)=q$,  $t_0=1/2$ and $\gamma(1/2)=\pi_1(b)\in B_M$, so $(x,y)$ is midpoint secure in M.

Conversely if $2)$ holds all secure configurations of $M$ and $S^2$ are midpoint secure, and $1)$
follows
from Lemma \ref{lemmaBG} and Lemma \ref{lemma1}.
\end{proof}

Note that this is a statement about $S^2$ as much as about the manifold $M$. The crucial property being
that there is a configuration (in this case (p,p)) where the midpoint blocking point is needed to block
$G(p,p)$. There are many manifolds that contain such configurations, for example $S^n$ with any metric
of revolution, but all known examples also have insecure configurations.\\


\section{Example}\label{secEx}

We will construct an example of a manifold that contains secure configurations that are not
midpoint secure.\\

Let $C$ be a cylinder of length $l$ and radius $1$, write it as a product of an interval $[0,l]$ and a
circle $S^1$ of radius $1$. Let $H_0$ and $H_l$ be a lower and an upper hemisphere, and attach them to
$C$ by identifying the equators with the the curves $0\times S^1$ and $l\times S^1$ respectively. We
get $N=C\cup H_0\cup H_l/\sim$ where $\sim$ is the identification above.\\

First we need to understand some of the geodesics on $N$. Observe that any geodesic in the cylinder
that reaches $l\times S^1$, forming an angle $\alpha$ with it, goes into $H_l$ where it is a half great
circle that leaves $H_l$ again at its antipodal point, forming the same angle $\alpha$. From the point
of view of the cylinder, any geodesic that leaves it through a point $(l,\theta)$ comes back at the
point $(l,\theta+\pi)$ with the same angle.

Let $\hat{C}=[-l,-0]\times S^1$ be a reflection of $C$, where we denote $0$ by $-0$ to distinguish the
points in $\hat{C}$ from those in $C$. Let $T=C\cup\hat{C}/\sim$ be the torus formed by gluing both
cylinders with a $1/2$ twist, i.e. $(l,\theta)\sim(-l, \theta+\pi)$ and $(0,\theta)\sim(-0,
\theta+\pi)$. Let $p:T\to C$ be the projection $p(t,\theta)=(|t|,\theta)$ for $t\neq -0,-l$.  By the
argument above, the restriction to $C$ of a geodesic in $N$ that goes through a hemisphere and back to
$C$ is the projection under $p$ of a geodesic in $T$ that goes from $C$ to $\hat{C}$.

The geodesics in $N$ with endpoints in the interior of $C$ corresponds to the projections to $C$
of the geodesics in $T$, although the parametrizations usually do not agree.\\

\begin{lemma} Any pair of points in the interior of $C$ form a secure configuration in $N$.
\end{lemma}

\begin{proof}For any pair of points $x_1=(t_1, \theta_1)$ and $x_2=(t_2, \theta_2)$ in the interior of $C$ and any
geodesic $\gamma$ between them in $N$, let $\tilde{\gamma}$ be the geodesic in $T$ that projects to
$\gamma$ with $\tilde{\gamma}(0) =x_1$. By the arguments above the projection of $\tilde{\gamma}$ is
$\gamma \cap C$ so any point that blocks $\tilde{\gamma}$ in $T$ will project to a point that blocks
$\gamma$.

On the other hand, since $\tilde{\gamma}(1)$ projects to $\gamma(1)=x_2$ it is either $x_2$ or
$-x_2=(-t_2, \theta_2)$ and we can identify $G(x_1,x_2)$ with the union of the sets $G_T(x_1,x_2)$ and
$G_T(x_1,-x_2)$ of geodesics in $T$. Since $T$ is a flat torus it is known that it is uniformly secure
with security threshold $4$ (see e.g. \cite{Gu04}). Therefore the set $B$ given by the projection of the $8$
blocking points of $G_T(x_1,\pm x_2)$ in $T$ will block $G(x_1,x_2)$.
\end{proof}

\bigskip
\begin{lemma}Any pair of points in the interior of $C$ form a configuration that is not midpoint secure in $N$.
\end{lemma}

\begin{proof} For any pair of points $x_1=(t_1, \theta_1)$ and $x_2=(t_2, \theta_2)$ in the interior of $C$ let
$\tilde{\gamma}_i$ be the geodesic in $T$ from $x_1$ to $-x_2=(-t_2, \theta_2)$ that goes $i$ times
around in the $S^1$ direction and doesn't cross $l\times S^1$, it has length
$\tilde{l_i}=\sqrt{(t_1+t_2)^2+(2\pi i +\theta_2+\pi -\theta_1)^2}$. Let $\gamma_i$ be the
correspondent geodesic in $N$, it begins and ends in $C$ and goes through $H_0$ once so it has length
$l_i=\tilde{l_i}+\pi$.

If $t_1\neq t_2$ we can assume that $t_1> t_2$.  The time that $\gamma_i$ spends in $H_0$, when
$\gamma_i$ is parametrized by $[0,1]$, is an interval of length $\pi/l_i$ beginning at
$(1-\frac{\pi}{l_i})\frac{t_1}{t_1+t_2}$. Since $t_1/(t_1+t_2)>1/2$ and the length $l_i$ grows to
infinity with $i$, for $i$ big enough $(1-\frac{\pi}{l_i})\frac{t_1}{t_1+t_2}>\frac{1}{2}$ so
$\gamma_i$ spends more than half the time before reaching $H_0$. The midpoint is then
$\gamma_i(1/2)=p(\tilde{\gamma}_i(\tilde{l_i}/2l_i))$ that has $t$ coordinate
$(t_1+t_2)\tilde{l_i}/2l_i=(t_1+t_2)(1/2 - \pi/2l_i)$, clearly different. If this configuration is
midpoint secure all this points have to be in the blocking set, making it infinite.

In the case that $t_1=t_2$ the midpoints $\gamma_i(1/2)$ are the midpoints of the restriction
of $\gamma_i$ to $H_0$. We can see that the distance between $\gamma_i(1/2)$ and the equator depends
directly on the angle that $\gamma_i$ makes with it. This angle gets smaller as $i$ increases,
showing that the points $\gamma_i(1/2)$ are all distinct.
\end{proof}

\begin{rem}
This manifold is not secure, for example the configuration $((0,0),(0,\pi))$ is insecure since there
are infinitely many disjoint paths joining them in $H_0$. This leaves open whether all secure manifolds
are also midpoint secure.

Also, as mentioned above, it is only a $C^1$ manifold. It would be interesting to find a smooth example.\\
\end{rem}

From the discussion above we see that the statement 2) of Proposition \ref{prop} doesn't hold.
Therefore there are insecure configurations in $N\times S^2$ that project to secure configurations in
both factors. As pointed out previously, this gives an explicit example of an insecure product of
secure configurations.

For this particular manifold $N$ we can prove the stronger statement below.\\

\begin{prop} For any Riemannian manifold $M$, any configuration in the product $N\times M$ that projects to one of the above configurations is
insecure.
\end{prop}
\begin{proof} Suppose not, then there is a secure configuration $(x_1,y_1),(x_2,y_2)$ where $x_1$ and $x_2$ are in
the interior of $C$. Let $B$ be a blocking set for this configuration, and $B_M$, $B_N$ its projection
to each factor. Let $\sigma$ be any geodesic in $M$ joining $y_1$ and $y_2$, and let $t_1,\dots t_n$ be
all the times where $\sigma(t_i)\in B_M$. Let $\gamma_i$ be as above, the set of times when
$\gamma_i(t)\in B_N$ is different for each $i$, and we can find an $i$ for which this times are all
different from $t_1,\dots t_n$. Then the geodesic $\gamma_i\times \sigma$ is not blocked
by $B$, giving a contradiction.
\end{proof}


\section{Billiard Tables}\label{secBill}

The concept of security originated in the study of billiards. In this subject a billiard table is a
manifold with boundary, and the trajectories of billiard balls follow geodesics and bounce of the
boundary according to the usual reflection laws. It is particularly related with the illumination
problem that studies which points does a light source illuminates, and which are shaded, when the
boundary is considered as a perfect mirror.

Security, also called finite blocking property, has been studied mainly on flat tables, with smooth or
polygonal boundaries. The examples of secure billiard tables are as limited as the examples of secure
manifolds. For lattice polygonal billiards security is equivalent to being arithmetic (see
\cite{Gu04}) and, to the best of my knowledge, no other examples are known. In fact, it was recently
proved by S. Tabachnikov in \cite{T} that any planar billiard with smooth boundary is insecure.

In this respect, we can consider half of the surface $N$ (with boundary) as a billiard table, cutting
it in either of the two natural ways. This gives examples of billiard tables with many secure
configurations. These tables are not planar and they have geodesic boundary, although the surfaces are
only $C^1$.\\

Let $\tau$ be the closed geodesic that agrees with $(t,0)$ and $(t,\pi)$ in the cylinder and the
corresponding half circles in $H_0$ and $H_l$. And let $M_1$ be the half of $N$ bounded by it (a half
cylinder with quarter spheres attached at the ends).

We can build $N$ from $M_1$ by gluing it with a reflection $\hat{M}_1$ of itself along the boundary, so
we can say that $N$ is a double cover of $M_1$ in a similar way that $T$ is a double cover of $C$. This
cover preserves geodesics, since a geodesic crossing from $M_1$ to $\hat{M}_1$ projects to a geodesic
that bounces off $\tau$. Conversely, any geodesic in $M_1$ can be lifted to $N$, and if it bounces off
the boundary the lift crosses from $M_1$ to its reflection. Therefore this table has similar blocking
properties as $N$, any geodesic blocked in $N$ projects to a geodesic in $M_1$ that is blocked
by the projection of the original blocking point.\\

Let $\sigma$ be the circle $l/2\times S^1$, and $M_2=\{(t,\theta), t\leq l/2\}\cup H_0$ a cylinder with
a half sphere at one end, and boundary $\sigma$. The same argument we used for $M_1$ shows that $N$ is
a geodesic double cover of $M_2$, and thus has similar blocking properties. In particular we have
proved the following lemma.\\

\begin{lemma}
Any pair of points in the flat part of $M_1$ or $M_2$ is secure, and not midpoint secure.

Moreover, any pair of points in the boundary of $M_2$ form a secure configuration.
\end{lemma}

Unlike $M_1$ the boundary of $M_2$ is smooth, therefore we have a billiard table with smooth boundary
such that all pairs of points in the boundary can be blocked.  This can't be achieved by smooth planar
billiards, as shown in \cite{T} where insecurity is proved for points in the boundary.\\

If we allow manifolds with corners in the boundary, we can consider a quarter of $N$ cutting it along
both $\tau$ and $\sigma$. Then $N$ is a geodesic 4-fold cover of this region, and therefore they share
blocking properties as with $M_1$ and $M_2$. The same is true for the region bounded by $\tau$ and
$\tau_n$, or by $\tau$, $\tau_n$ and $\sigma$, where $\tau_n$ (for $n\geq 2$)  is the geodesic between
the north and south poles that agrees with $(t,\pi/n)$ in the cylinder.


\end{document}